\documentclass[12pt]{article}
\usepackage{amssymb}
\usepackage{amsfonts}
\title{Classification of some countable descendant-homogeneous digraphs}

\author{Daniela Amato, David M. Evans, University of East Anglia \\ and \\
John K. Truss, University of Leeds.}
\date{}
\begin{document}
\maketitle 
\newtheorem{lemma}{Lemma}[section]
\newtheorem{theorem}[lemma]{Theorem}
\newtheorem{corollary}[lemma]{Corollary}
\newtheorem{definition}[lemma]{Definition}
\newtheorem{note}[lemma]{Definition}
\newtheorem{remark}[lemma]{Remark}
\newtheorem{proposition}[lemma]{Proposition}
\newtheorem{conjecture}[lemma]{Conjecture}
\setcounter{footnote}{1}\footnotetext{This work was supported by EPSRC grant EP/G067600/1}
\setcounter{footnote}{2}\footnotetext{5 April 2011}
\newcounter{number}

\def\Aut{{\rm Aut}}
\def\C{\mathcal{C}}
\def\desc{{\rm desc}}
\def\N{\mathbb{N}}

\newenvironment{proof}[1][Proof]{\textbf{#1.} }{\ \rule{0.5em}{0.5em}}
\begin{abstract}
\noindent For finite $q$, we classify the countable, descendant-homogeneous digraphs in which the descendant set of any vertex is a $q$-valent tree. We also give   conditions on a rooted digraph $\Gamma$ which allow us to construct a countable descendant-homogeneous digraph in which the descendant set of any vertex is isomorphic to $\Gamma$.
 

\end{abstract}

\renewcommand{\thetheorem}{\thesection.\arabic{theorem}}

\section{Introduction}
\subsection{ Background}

A countable digraph is homogeneous if any isomorphism between finite (induced) subdigraphs extends to an automorphism. The digraphs with this property 
are classified by Cherlin  in  \cite{cherlin}. By analogy, the notion of descendant-homogeneity was introduced in \cite{amato4}. A countable digraph 
is {\em descendant-homogeneous} if any isomorphism between finitely generated subdigraphs extends to an automorphism. Here, a subdigraph is finitely 
generated if its vertex set can be written as the descendant set of a finite set of vertices, that is, the set of vertices which are  reachable by a 
directed path from the set. 

Note that descendant-homogeneity can hold for trivial reasons: digraphs where the descendant set of any vertex is the whole 
digraph, or where no two vertices have isomorphic descendant sets are descendant-homogeneous. So it is reasonable to impose  restrictions such as 
vertex transitivity and no directed cycles. We refer to \cite{amato4} for further discussion. 

In this paper we are particularly interested in   vertex-transitive, descendant-homogeneous digraphs: so in this case, the descendant set of any vertex is isomorphic to some fixed digraph $\Gamma$. Examples of countable,  vertex-transitive, descendant-homogeneous digraphs where $\Gamma$ is a $q$-valent directed tree (for finite $q >1$)  were given in \cite{evans1}, \cite{amato4}. The main result of this paper is to show that the digraphs constructed in \cite{evans1} and 
\cite{amato4} constitute a complete list of all the countable descendant-homogeneous digraphs with descendant sets of this form (Theorem 
\ref{maintheorem}). In the final section of the paper, we give general conditions on $\Gamma$ under which there is a countable, vertex-transitive, descendant-homogeneous digraph in which the descendant set of any vertex is isomorphic to $\Gamma$. In particular, these conditions are satisfied by certain `tree-like' digraphs $\Gamma$ studied in \cite{amato1}. This gives new examples of descendant-homogeneous digraphs (and indeed, highly arc-transitive digraphs).
 
The first (non-trivial) examples of descendant-homogeneous digraphs known to the authors arose in the context of highly arc-transitive digraphs (those 
whose automorphism groups are transitive on the set of $s$-arcs for all $s$). In answer to a question of Cameron, Praeger, and Wormald in 
\cite{cameron}, the paper \cite{evans1} gave a construction of a certain highly arc-transitive digraph $D$ having an infinite binary tree as 
descendant set. The digraph was constructed as an example of a highly arc-transitive digraph not having the `property Z', meaning that there is no 
homomorphism from $D$ onto the natural digraph on $\mathbb Z$ (the doubly infinite path). However, we noted in \cite{amato4} that it is also 
descendant-homogeneous.  A more systematic analysis of this notion was carried out in \cite{amato4}, and further examples were given. The method of 
\cite{evans1} immediately applies to $q$-valent trees for any finite $q > 1$ in place of binary trees, but in addition, it is shown in \cite{amato4} 
that it is possible to omit certain configurations and still carry out a Fra\"{\i}ss\'e-type construction to give other examples of 
descendant-homogeneous digraphs whose descendant sets are $q$-valent trees. The classical Fra\"{\i}ss\'e theorem for relational structures provides a 
link between countable homogeneous structures (those in which any isomorphism between finite substructures extends to an automorphism) and 
amalgamation classes of finite structures. See \cite{cameron1}, \cite{cherlin} and \cite{hodges} for instance. The analogue of Fra\"{\i}ss\'e's  Theorem and the appropriate notion of amalgamation classes 
 which applies to descendant-homogeneity is given in Section \ref{FT}.

\subsection{Notation and Terminology}
Let $D$ a digraph with vertex and edge sets $VD$ and $ED$, and let $u \in VD$. For $s\geq 0$, an $s$\textit{-arc} in $D$ from $u_0$ to $u_s$  is a 
sequence $u_0u_1 \ldots u_s$ of $s+1$ vertices such that $(u_i, u_{i+1})\in ED$ for $0 \leq i < s$ and $u_{i-1} \neq u_{i+1}$ for $0 < i < s$. We let
\[
{\rm desc}^s(u) :=\{v\in VD\mid \mbox{there is an } s \mbox{-arc from }u \mbox{ to } v \},
\]
and ${\rm desc}(u)=\bigcup_{s\geq 0} {\rm desc}^{s}(u)$, the {\em descendant set} of $u$ (we also denote this by ${\rm desc}_D(u)$ if we need to 
emphasize that we are looking at descendants in $D$). If $X \subseteq VD$, we also let
\[
{\rm desc}^{s}(X):=\bigcup_{x\in X}{\rm desc}^{s}(x),  
\]
and similarly ${\rm desc}(X):=\bigcup_{x\in X} {\rm desc}(x)$. The `ball' of radius $s$ at $u$ is given by 
\[
B^s(u):=\bigcup_{0\leq i\leq s}{\rm desc}^{i}(u).
\]
 \smallskip

For a digraph $D$ we often write $D$ in place of $VD$ and use the same notation for a subset of the vertices and the full induced subdigraph. Henceforth, `subdigraph' will mean `full induced subdigraph' and an embedding of one digraph into another will always mean as a full induced subdigraph.  

We say that $A \subseteq D$ is {\itshape descendant-closed} in $D$, written $A \leq D$ if 
${\rm desc}_D(a) \subseteq A$ for all $a \in A$; and we say that an embedding $f:A \rightarrow B$ between digraphs is a 
$\leq$-{\itshape embedding} if $f(A)\leq B$. When  $A,B_1, B_2$ are digraphs we say that   $\leq$-embeddings $f_i:A \rightarrow B_i$  are {\itshape isomorphic} if there is an isomorphism $h: B_1\rightarrow B_2$ with $f_2=h \circ f_1$.

We say that $A \leq D$ is \textit{finitely generated} if there is a finite subset $X \subseteq A$ with $A = \desc_D(X)$; in this case we refer to $X$ as a generating set of $A$. If additionally no proper subset of $X$ is a generating set, then $X$ is called a minimal generating set. Clearly, in this case, no element in $X$ is a descendant of any other element of $X$.

The digraph $D$ is \textit{descendant-homogeneous} if whenever $f: A_1 \to A_2$ is an isomorphism between finitely generated descendant-closed subdigraphs of $D$, there is an automorphism of $D$ which extends $f$. The group of automorphisms of $D$ is denoted by $\Aut(D)$.

We shall 
mainly be concerned with digraphs $D$ where the descendant sets of single vertices are all isomorphic to a fixed digraph $\Gamma$: in this case we refer to 
$\Gamma$ as `the descendant set' of $D$.  A 
subset of a digraph is {\em independent} if the descendant sets of any two of its members are disjoint. In any digraph in which the descendant sets 
are all isomorphic, for any two finite independent subsets $X$ and $Y$, any bijection from $X$ to $Y$ extends to an isomorphism from ${\rm desc}(X)$ 
to ${\rm desc}(Y)$ since ${\rm desc}(X)$ and ${\rm desc}(Y)$ are both the disjoint union of $\vert X\vert$ descendant sets. 

Throughout we fix an integer $q > 1$ and write $T = T_q$ for the $q$-valent rooted tree. So $T$ has as its vertices the set of finite sequences from 
the set $\{0,\ldots,q-1\}$ and directed edges $(\bar{w}, \bar{w}i)$ (for $\bar{w}$ a finite sequence and $i \in \{0,\ldots, q-1\}$).

\section{Amalgamation classes}
\subsection{The Fra\"{\i}ss\'e Theorem}\label{FT}

As in \cite{amato4}, the correct context for the study of descendant-homogeneous 
digraphs is a suitable adaptation of Fra\"{\i}ss\'e's notion of \textit{amalgamation classes}. The reader who is familiar with this type of result (or with \cite{amato4}) and who is mainly interested in the main classification result, Theorem \ref{maintheorem}, could reasonably skip to the next subsection. The extra generality which is given here is only needed in the final section of the paper.

Let $\cal D$ be a class of (isomorphism types of) digraphs. 
Then $\cal D$ has the $\leq$-{\itshape amalgamation property} if the following holds: if $A$, $B_1$ 
and $B_2$ lie in $\cal D$, and $\leq$-embeddings $f_1$ and $f_2$ of $A$ into each of $B_1$  and $B_2$ are given, then there are a structure $C \in 
{\cal D}$ and $\leq$-embeddings $g_1$ and $g_2$ of $B_1$ and $B_2$ respectively into $C$ such that $g_1\circ f_1 = g_2 \circ f_2$. We say that $g_1, g_2$ solve the amalgamation problem given by $f_1, f_2$. 

\begin{remark} \label{freedef}\rm Suppose $A$, $B_1$, and $B_2$ are digraphs and $\leq$-embeddings $f_1$ and $f_2$ of $A$ into each of $B_1$  and $B_2$ are given. We can clearly find a solution $g_i: B_i \to C$ with the property that $C = g_1(B_1)\cup g_2(B_2)$, $g_1(B_1) \cap g_2(B_2) = g_1(f_1(A))$ and every directed edge is contained  in $g_1(B_1)$ or $g_2(B_2)$. Moreover, this solution is uniquely determined up to isomorphism by the $f_i$. Informally, we can regard the $f_i$  as inclusion maps and take $C$ to be the disjoint union of  $B_1$ and $B_2$ over $A$. We make this into a digraph by taking as edge set  $EC = EB_1\cup EB_2$. It is easy to see that $B_1, B_2 \leq C$ and the inclusion maps $g_i :B_i \to C$ satisfy $g_1\circ f_1 = g_2 \circ f_2$. We say that the solution $g_i : B_i \to C$ to the problem $f_i : A \to B_i$ is the \textit{free amalgam} of the $f_i$. When $f_1, f_2$ are inclusion maps (or are understood from the context) we shall abuse this terminology and say that $C$ is the free amalgam of $B_1$ and $B_2$ over $A$.

Note that if $B_1, B_2 \leq C$ then $B_1 \cup B_2 \leq C$ and $B_1\cup B_2$ is the free amalgam of $B_1$ and $B_2$ over $B_1\cap B_2$: there can be no directed edges between elements of $B_1\setminus B_2$ and $B_2 \setminus B_1$ as $B_1, B_2$ are descendant-closed.

When we come to count structures and embeddings up to isomorphism (as in  Lemma \ref{countable}), it will be useful to have a more precise notation for free amalgamation. Suppose in the above that  $f_1$ is inclusion and $f_2$ is an arbitrary $\leq$-embedding $f_2:A \rightarrow B_2$.  The free amalgam  $B_1 {\ast}_{f_2} B_2$ has as vertex set the disjoint union of $B_1\setminus A$ and $B_2$ (and the `obvious' directed edges). The embedding $g_2 : B_2 \to B_1 \ast_{f_2} B_2$ is inclusion and the embedding  $g_1 : B_1 \to B_1\ast_{f_2} B_2$ is given by $g_1(b) = b$ if $b \in B_1 \setminus A$ and $g_1(b) = f_2(b)$ if $b \in A$. 

We remark that in general, if $A \leq B_1$ and $f_2 , f_2' : A \to B_2$ are $\leq$-embeddings with the same image, then $B_1 \ast_{f_2} B_2$ and $B_1 \ast_{f_2'} B_2$ need not be isomorphic.
\end{remark}

The analogue of Fra\"{\i}ss\'e's Theorem which we use is the following.

\begin{theorem} \label{2.1} Suppose $M$ is a countable descendant-homogeneous digraph. Let ${\cal C}$ be the class of digraphs which are isomorphic to finitely generated $\leq$-subdigraphs of $M$. Then 
\begin{enumerate}
\item[{\rm (1)}]  ${\cal C}$ is a class of countable, finitely generated digraphs which is closed under isomorphism and  has countably many isomorphism types; 
\item[{\rm (2)}] ${\cal C}$ is closed under taking finitely generated $\leq$-subdigraphs;
\item[{\rm (3)}] ${\cal C}$ has the $\leq$-amalgamation property;
\item[{\rm (4)}] for all $A,B \in {\cal C}$ there are only countably many isomorphism types of $\leq$-embeddings from $A$ to $B$.\end{enumerate}
Conversely, if ${\cal C}$ is a class of  digraphs satisfying (1)-(4), then there is a countable descendant-homogeneous digraph $M$ for which the class of digraphs isomorphic to finitely generated $\leq$-subdigraphs of $M$ is equal to  ${\cal C}$. Moreover, $M$ is determined up to isomorphism by ${\cal C}$. 
 \end{theorem}

We refer to a class $\cal C$ of digraphs satisfying (1)-(4) as a $\leq$-amalgamation class. The digraph $M$ determined by $\cal C$ as in the theorem  is called the Fra\"{\i}ss\'e limit of $(\cal C, \leq)$. 

\begin{remark}\rm \label{rem2}
It is easy to see that in place of (4) we can substitute the condition: \textit{
\begin{enumerate}
\item[{ \rm(4$'$)}] if $A \leq B \in {\cal C}$ and $A$ is finitely generated, then the subgroup of the automorphism group ${\rm Aut}(A)$ consisting of automorphisms which extend to automorphisms of $B$ is of countable index in ${\rm Aut}(A)$. 
\end{enumerate}}
Indeed, we wish to consider the number of $\leq$-embeddings $f : A \to B$ up to isomorphism. As $B$ is countable and $A$ is finitely generated there are countably many possibilities for the image $f(A)$, so it will be enough to count isomorphism types of $\leq$-embeddings  with fixed finitely generated image $Y \leq B$. Let $H$ be the subgroup of $\Aut(Y)$ consisting of automorphisms which extend to automorphisms of $B$. It is straightforward to show that if $f, f' : A \to B$ have image $Y$, then $f, f'$ are isomorphic if and only if the map $g \in \Aut(Y)$ given by $g(y) = f'(f^{-1}(y))$ is in $H$. Thus there is a bijection between the $H$-cosets in $\Aut(Y)$ and the isomorphism types.
\end{remark}

\begin{remark}\rm The proof of Theorem \ref{2.1} is reasonably standard, but we make some comments on the condition (4). First, suppose $M$ and $\C$ are as in the statement. We show that (4$'$) in Remark \ref{rem2} holds. So let $A \leq B \in \C$ and $H \leq \Aut(A)$ be the automorphisms of $A$ which extend to automorphisms of $B$, as in (4$'$). We may assume $B \leq M$. Suppose $g_1, g_2 \in \Aut(A)$ lie in different $H$-cosets. As $M$ is $\leq$-homogeneous we can extend $g_i$ to $k_i \in \Aut(M)$. Then $k_1(B) \neq k_2(B)$. Otherwise $h = k_2^{-1}k_1$ stabilizes $B$ and gives an automorphism of $B$ which extends $h = g_2^{-1}g_1$; this implies $h \in H$ and $g_2H = g_1H$, which is a contradiction. As there are only countably many possibilities for the image of $B$ under automorphisms of $M$, it follows that $H$ is of countable index in $\Aut(A)$, as required.

The converse is a fairly standard construction, and can be read off from from Theorem 2.18 of \cite{kirby}, which in turn is adapted from Theorem 1.1 of \cite{DG}. However, we give a few details of the proof. So suppose we have a class $\C$ of finitely generated digraphs satisfying (1)-(4). We construct a countable chain $C_1 \leq C_2 \leq C_3 \leq \ldots$ of digraphs in $\C$ with the property that if $A \leq C_i$ is finitely generated and $f : A \to B \in \C$ is a $\leq$-embedding, then there is $j \geq i$ and a $\leq$-embedding $g : B \to C_j$ with $g(f(a)) = a$ for all $a \in A$. The resulting digraph $\bigcup_i C_i$ will be descendant-homogeneous, by a back-and-forth argument. Note that by (4), we have only countably many $f$ to consider (for any particular $A$). For if $f,g$ are as above and $f': A \to B$ is isomorphic to $f$ with $f' = h\circ f$ for $h \in \Aut(B)$, then $g' = g\circ h^{-1} : B \to C_j$ satisfies $g'(f'(a)) = a$ for all $a \in A$.
\end{remark}

\subsection{The classification result}

Recall that $q \geq 2$ is an integer and $T = T_q$ is the $q$-valent rooted tree. We shall classify countable, descendant-homogeneous digraphs $M$ in which the descendant sets of vertices are isomorphic to $T$. Thus, by Theorem \ref{2.1}, we need to classify $\leq$-amalgamation classes of finitely generated digraphs with descendant sets isomorphic to $T$.  In this case,  we can replace the condition (4) in Theorem \ref{2.1} by the simpler condition:
\textit{
\begin{enumerate}
\item[{\rm (4$''$)}] if $a_1,a_2 \in B\in {\cal C}$, then ${\rm desc}_B(a_1) \cap {\rm desc}_B(a_2)$ is finitely generated
\end{enumerate}}
\noindent as in Theorem 3.4 of \cite{amato4}. Indeed, if ${\cal C}$ satisfies (4$''$) then (4$'$) is a special case of Lemma \ref{countable} here. Conversely, if (4$'$) holds, then to see (4$''$) let $B={\rm desc}(a_1) \cap {\rm desc}(a_2)$ and $A={\rm desc}(a_1)$. Let $X$ be the minimal generating set for $A \cap {\rm desc}(a_2)$. Then $X$ is independent and any automorphism of $B$ which stabilizes $A$ must fix $X$ setwise. On the other hand, if $Z$ is an infinite independent subset of $A$ it is easy to see that the stabilizer of $Z$ in ${\rm Aut}(A)$ is of index continuum (as there are continuum many translates of $Z$ by automorphisms of $A$, since $A$ is a regular rooted tree).

\medskip

Thus we work with the class $\cal C= {\cal C}_{\infty}$ consisting of  all digraphs $A$ satisfying the following conditions:  
 \begin{itemize}
 \item for all $a \in A$, ${\rm desc}(a)$ is isomorphic to $T$; 
\item $A$ is finitely generated;
\item for $a,b \in A$, the intersection ${\rm desc}(a)\cap {\rm desc}(b)$ is finitely generated.
\end{itemize}

Then $\cal C$ satisfies conditions (1), (2), (4) in Theorem \ref{2.1} (cf. the above remarks and Lemma \ref{countable}), and we are interested in the subclasses of $\C$ which satisfy (3).  It is easy to see that $\C$ satisfies $(3)$: in fact $\cal C$   is closed under  free 
amalgamation. It follows that $({\cal C}, \leq)$ is a $\leq$-amalgamation class. The Fra\"{\i}ss\'e limit $D_{\infty}$ of $(\cal C, \leq)$ is the countable descendant-homogeneous digraph  constructed  in \cite{evans1}. 

\smallskip

For $n \geq 2$, we now define the amalgamation classes ${\cal C}_n \subseteq {\cal C}$ (from \cite{amato4}). Let ${\cal T}_n$ be the element of $\cal C$ generated by $n$ elements $x_1, \ldots, x_n$, such that ${\rm desc}^1(x_i) = 
{\rm desc}^1(x_j)$ for all $i \neq j$. So ${\cal T}_n$ is like the tree $T$, except that there are $n$ root vertices (all having the same 
out-vertices). Let ${\cal C}_n$ consist of the digraphs $A \in {\cal C}$ such that  ${\cal T}_n$ does not embed in $A$ (as a descendant-closed 
subdigraph).

It is clear that ${\cal C}_n \subseteq {\cal C}_{n+1}$ and ${\cal C}_n \subseteq {\cal C}$ for all $n$. In \cite{amato4} it is shown that 
$({\cal C}_n, \leq)$ is a $\leq$-amalgamation class, though it is clearly not a free 
amalgamation class. In particular, when we `solve' an amalgamation problem $f_i : A \to B_i$ by maps $g_i : B_i \to C$, we may have 
$g_1(B_1) \cap g_2(B_2) \supset g_1(f_1(A))$. Informally, this means that  points of $B_1$, $B_2$ outside $A$ may need to become identified in the 
amalgam $C$.

For $n \geq 2$, let $D_n$ be the Fra\"{\i}ss\'e limit of  $({\cal C}_n,\leq)$, as in Theorem \ref{2.1}. Then $D_n$ is a countable 
descendant-homogeneous digraph. Our main result is:
\begin{theorem}\label{maintheorem}
Let $D$ be a countable descendant-homogeneous digraph whose descendant sets are isomorphic to $T$. Then $D$ is isomorphic to $D_n$ for some $n \in \{2, \ldots, \infty\}$. 
\end{theorem}

\bigskip

\section{Proof of the main theorem}

We know from \cite{amato4} that each ${\cal C}_n \subseteq {\cal C}$ is a $\leq$-amalgamation class. From now on we shall consider an arbitrary 
subclass $\cal D$ of $\cal C$ which is itself a $\leq$-amalgamation class (that is, satisfies (1)-(4) of Theorem \ref{2.1}), with the goal of showing 
that $\cal C$ and ${\cal C}_n$ are the only possibilities for $\cal D$. 

To understand the argument better, suppose that there is some integer $n \geq 2$ such that ${\cal T}_n \notin {\cal D}$. Choose $n$ as small as 
possible: so in particular,  ${\cal T}_{n-1} \in {\cal D}$ (where ${\cal T}_1 = T$) and  ${\cal D} \subseteq {\cal C}_{n}$. To prove our main result 
it suffices to show that if $A \in {\cal C}_n$ then $A \in {\cal D}$, and this is done by induction on the number of generators of $A$. Let $\{a_1, 
\ldots a_k\}$ be the minimal generating set of $A$ and let $A_1$ be the descendant-closed subdigraph of $A$ with generating set $\{a_1, \ldots, 
a_{k-1}\}$. Let $A_0 = A_1 \cap {\rm desc}(a_k)$. Then $A$ is the  free amalgam of $A_1$ and ${\rm desc}(a_k)$ over  $A_0$. By the induction 
hypothesis, $A_1 \in {\cal D}$, and we know  that ${\rm desc}(a_k) \cong T\in {\cal D}$. So there are $C \in {\cal D}$ and $\leq$-embeddings $f: A_1 
\rightarrow C$ and $g:T \rightarrow C$ such that $f(a)=g(a)$ for all $a \in A_0$ (identifying ${\rm desc}(a_k)$ with $T$). However, {\itshape{a 
priori}} one cannot force $C$ to be the free amalgam. So we replace $A_1$ by some $B \geq A_1$, $T$ by $T' \geq T$ and $A_0$ by $A_0' \leq B,T'$ in 
such a way that the amalgam in ${\cal D}$ of $B$ and $T'$ over $A_0'$ is forced to be free. This is the point of Lemmas 4.1, 4.2 and 4.3 (which do not 
need the extra assumption on ${\cal D}$).

\begin{lemma}\label{4.1}
Let $A \in {\cal C}$ and $X$ be a finite independent subset of $A$. Then there is a finite independent subset $Y$ of $A$ containing $X$ such that 
$A \setminus {\rm desc}(Y)$ is finite. 
\end{lemma}
\begin{proof}
Let $a,x \in VA$ and let $S$ be the minimal generating set of ${\rm desc}(a) \cap {\rm desc}(x)$. Since $S$ is finite, there is $n(a,x) \in 
{\mathbb N}$ such that $S \subseteq B^{n(a,x)}(a)$. Let $m \geq n(a,x)$, $y \in {\rm desc}^m(a)$ and $y \notin {\rm desc}(x)$. Then ${\rm desc}(y) 
\cap {\rm desc}(x) = \varnothing$: if not, let $u \in {\rm desc}(y) \cap {\rm desc}(x)$. As $y \in {\rm desc}(a)$, $u \in {\rm desc}(a) \cap {\rm 
desc}(x)$, so $u \in {\rm desc}(s)$ for some $s \in S$. As ${\rm desc}(a)$ is a tree, and by the choice of $m$, $y \in {\rm desc}(s) \subseteq {\rm 
desc}(x)$, which is a contradiction.

Let $a_1, \ldots, a_r$ be the minimal generating set for $A$. Now let $N \geq max\{n(a_i, a_j), \linebreak n(a_i, x) \mid i\neq j, x \in X\}$ and let 
$B=\bigcup_{i=1}^r B^N(a_i)$. So if $y \in A \setminus \left( B \cup {\rm desc}(X)\right)$ then ${\rm desc}(y) \cap {\rm desc}(X) =\varnothing$ and if 
$y_1, y_2 \in A \setminus \left( B \cup {\rm desc}(X)\right)$, and neither is a descendant of the other, then ${\rm desc}(y_1)\cap {\rm desc}(y_2) = 
\varnothing$. Let $y_1, \ldots, y_t$ be the maximal elements of $A \setminus \left( B \cup {\rm desc}(X)\right)$. Then $Y:= X \cup \{y_1, \ldots, 
y_t\}$ is independent and $A \setminus {\rm desc}(Y) \subseteq B$ is finite. 
\end{proof}
\bigskip

For a finite independent subset $X$ of $A$, and $Y$ given by the lemma, we say that $Y \setminus X$ is a {\itshape{complement}} of $X$ in $A$. 

\medskip

For $X \subseteq D \in {\cal C}$, a \textit{common predecessor} for $X$ in $D$ is a vertex $a \in D$ such that $(a,x)$ is a directed edge for all $x \in X$. 
Let $A \in {\cal C}$ and let $U, V$ be independent subsets of $A$ and $T$ respectively with $f: {\rm desc}(U) \rightarrow {\rm desc}(V)$ an 
isomorphism. Let $Q$ be the set consisting of those $q$-element subsets $p$ of $U$ such that $p$ has a common predecessor in $A$ and $f(p)$ has a 
common predecessor in $T$. For $p \in Q$, let $w_p$ and $w_{f(p)}$  be such common predecessors of $p$ and $f(p)$ respectively. We note that as $T$ 
is a tree, $w_{f(p)}$ is uniquely determined, but $w_p$ may not be. Also, as $T$ is a tree, any two members of $Q$ are disjoint. Now let 
\[
U':= \left( U \setminus \bigcup Q \right) \cup \{w_p\mid p \in Q\} \;\;\; \mbox{and} \;\;\; 
V':=\left( V \setminus \bigcup f(Q) \right) \cup \{w_{f(p)}\mid p \in Q\}
\]
In words, $U'$ is obtained from $U$ by replacing the vertices in $p \subseteq U$ by their common predecessors $w_p$, for all $p \in Q$. Similarly $V'$ is obtained from $V$. Clearly $|U'| = |V'|$, ${\rm desc}(U) \subseteq {\rm desc}(U')$ and ${\rm desc}(V) \subseteq {\rm desc}(V')$.  Moreover, 

\begin{lemma}\label{4.2}\smallskip
{\rm (a)} The sets $U'$ and $V'$ are independent subsets of $A$ and $T$ respectively, and the extension $F$ of $f$ which takes $w_p$ to $w_{f(p)}$ for each $p 
\in Q$ is an isomorphism from ${\rm desc}(U')$ to ${\rm desc}(V')$; 

{\rm (b)} if $I \subseteq A$ is disjoint from $U$ and $U \cup I$ is an independent subset of $A$, then $U' \cup I$ is also independent. 
\end{lemma}
\begin{proof}
(a) Let $u_1$ and $u_2$ be distinct members of $U'$. If neither lies in $\{w_p \mid p \in Q\}$, then they are in $U$, so ${\rm desc}(u_1) \cap 
{\rm desc}(u_2) = \varnothing$ is immediate. Next suppose that $u_1 = w_p$ for $p \in Q$ and $u_2 \not \in \{w_p \mid p \in Q\}$. Then 
${\rm desc}(u_1) = \{u_1\} \cup {\rm desc}(p)$, and as $U$ is independent, ${\rm desc}(u) \cap {\rm desc}(u_2) = \varnothing$ for each $u \in p$, and 
also $u_1 \not \in {\rm desc}(u_2)$, and it follows that ${\rm desc}(u_1) \cap {\rm desc}(u_2) = \varnothing$. Finally, if $u_1 = w_{p_1}$ and $u_2 = 
w_{p_2}$, then ${\rm desc}(u_1) = \{u_1\} \cup {\rm desc}(p_1)$ and ${\rm desc}(u_2) = \{u_2\} \cup {\rm desc}(p_2)$. Now for each $u \in p_1$ and $u' \in p_2$, ${\rm desc}(u) \cap {\rm desc}(u') = \varnothing$ by the independence of $U$, and $u_1 \not \in {\rm desc}(p_2)$ and $u_2 \not \in 
{\rm desc}(p_1)$ are clear, from which it follows that ${\rm desc}(u_1) \cap {\rm desc}(u_2) = \varnothing$. This shows that $U'$ is independent, and 
the proof that $V'$ is independent is similar. 

To see that $F$ is an isomorphism, note that the only new points in its domain are $w_p$, and $F$ maps $w_p$ to $w_{f(p)}$, and 
$f({\rm desc}^1(w_p)) = {\rm desc}^1(w_{f(p)})$. 

\smallskip

(b) Since ${\rm desc}(w_p) = \{w_p\} \cup {\rm desc}(p)$ for each $p \in Q$, and $p \subseteq U$ and $U \cup I$ is independent, it follows that 
${\rm desc}(w_p) \cap {\rm desc}(x) = \varnothing$ for all $x \in I$, so $U' \cup I$ is also independent. 
\end{proof}

\begin{lemma}\label{4.3}
Let $A \in {\cal D}$ and let $U$ be a finite independent subset of $A$. Let $M$ be the maximal number of common predecessors in $A$ of $q$-element 
subsets of $U$, and let $N \geq M$ be such that ${\cal T}_N \in {\cal D}$. Then there is $B \in {\cal D}$ with $A \leq B$ and such that every 
$q$-element subset of $U$ has at least $N$ common predecessors in $B$. 
\end{lemma}

\begin{proof}
Let $P = \{p_1, \ldots, p_t\}$ be the set of all $q$-element subsets of $U$. (Note that, unlike in the previous proof, the members of $P$ need not be 
pairwise disjoint.) We construct a sequence $B_0 \leq B_1 \leq B_2 \leq \ldots 
\leq B_t$ in $\cal D$, such that $p_i$ has at least $N$ common predecessors in $B_l$ for all $i \leq l$ and $l \le t$. We start with $B_0: = A$ and  assume inductively that we have constructed $B_l$, where $l < t$. Let 
$p_{l+1} = \{u_{1}, \ldots, u_{q}\}$ and consider a copy of ${\cal T}_N$ with generating set $G = \{g_1, \ldots, g_N\}$. Let ${\rm desc}^1(G) = 
\{h_1, \ldots, h_q\}$. Both sets $\bigcup_{j=1}^q{\rm desc}(u_{j})$ and $\bigcup_{j=1}^q{\rm desc}(h_j)$ are the union of $q$ disjoint copies of 
$T$, so there is an isomorphism taking the first to the second such that $u_{j}$ is sent to $h_j$ for each $j$. Let $B_{l+1}$ be an amalgam in 
$\cal D$ of $B_l$ and ${\cal T}_N$ with $\bigcup_{j=1}^q{\rm desc}(u_{j})$ and $\bigcup_{j=1}^q{\rm desc}(h_j)$ identified by this isomorphism 
(since $B_l, {\cal T}_N \in {\cal D}$). We note that $p_{l+1}$ has at least $N$ common predecessors in $B_{l+1}$ since $\{h_1, \ldots, h_q\}$ has $N$ 
common predecessors in ${\cal T}_N$. Hence $B = B_t$ is a member of $\cal D$ as required. \end{proof}

\begin{proposition}
Let ${\cal D} \subseteq {\cal C}$ be a $\leq$-amalgamation class and suppose that ${\cal T}_m \notin {\cal D}$ for some $m \geq 2$. Then ${\cal D} = 
{\cal C}_n$ where $n$ is the least  m such that ${\cal T}_m \notin {\cal D}$.
\end{proposition}

\begin{proof}
Note that ${\cal D} \subseteq {\cal C}_n$ and ${\cal T}_{n-1} \in {\cal D}$. We shall show that ${\cal C}_n \subseteq {\cal D}$. Let 
$A \in {\cal C}_n$. We use induction on the number of generators of $A$ to show that $A \in {\cal D}$. Let $a_1, \ldots, a_s$ be the (distinct) 
generators of $A$. If $s=1$, or if $A$ is the disjoint union of finitely many copies of $T$, then $A$ embeds in $T$ and therefore $A \in {\cal D}$, 
since $T \in {\cal D}$. Now let $s \geq 2$ and suppose that $E \in {\cal D}$ for all $E\in {\cal C}_n$ with at most $s-1$ generators. Let 
$A_1:=\bigcup_{i=1}^{s-1}{\rm desc}(a_i)$ and let $T$ be a copy of the $q$-valent tree with $b$ its root. The digraph $A$ is the free amalgam of $A_1$ 
and ${\rm desc}(a_s) \; (\cong T$) over $ A_1 \cap {\rm desc}(a_s)$ (which is finitely generated). So there are independent subsets $U = \{u_1, 
\ldots, u_k\}$ and $V = \{v_1, \dots, v_k\}$ of $A_1$ and $T = {\rm desc}(b)$ respectively and an isomorphism $f$ from ${\rm desc}(U)$ to 
${\rm desc}(V)$ (taking $u_i$ to $v_i$ for all $i$), such that $A$ is isomorphic to the free amalgam $C$ of $A_1$ and $T$ with ${\rm desc}(U)$ and 
${\rm desc}(V)$ identified by $f$. See Figure 1. To prove the result it then suffices to show that there is $D \in {\cal D}$ embedding $C$. We shall 
first `expand' $A_1$ to a digraph $B \in {\cal D}$ (using Lemma \ref{4.3}) and then amalgamate $B$ with a copy $T'\geq T$ of $T$ over the descendant 
sets of some carefully chosen independent subsets. The resulting digraph is then the required digraph $D$. 
 
\begin{figure}[htbp]
\begin{center}
\setlength{\unitlength}{5mm}
\linethickness{0.2mm}

\begin{picture}(10,7)(0,0)

\put(-3,6){\circle*{0.2}}
\put(-1,6){\circle*{0.2}}
\put(4,6){\circle*{0.2}}
\put(10,6){\circle*{0.2}}
\put(-3,7){\makebox(0,0)[t]{{\small $a_1$}}}
\put(-1,7){\makebox(0,0)[t]{{\small $a_2$}}}
\put(4,7){\makebox(0,0)[t]{{\small $a_{s-1}$}}}
\put(10,7){\makebox(0,0)[t]{{\small $b$}}}

\put(-3,6){\line(-1,-3){1.7}}
\put(-3,6){\line(1,-2){1}}
\put(-1,6){\line(-1,-2){1}}
\put(-1,6){\line(1,-3){1.7}}
\put(4,6){\line(-1,-3){1.7}}
\put(4,6){\line(1,-3){1.7}}
\put(10,6){\line(-1,-3){1.7}}
\put(10,6){\line(1,-3){1.7}}

\put(1,6){\circle*{0.1}}
\put(1.5,6){\circle*{0.1}}
\put(2,6){\circle*{0.1}}

\put(2,3){\circle*{0.12}}
\put(-0.5,3){\circle*{0.12}}
\put(0,3){\circle*{0.12}}
\put(0.5,3){\circle*{0.12}}
\put(1,3){\circle*{0.12}}
\put(1.5,3){\circle*{0.12}}
\put(0.6,2.7){\makebox(0,0)[t]{{\small $U$}}}

\put(9.2,3){\circle*{0.12}}
\put(9.5,2.7){\circle*{0.12}}
\put(9.8,2.4){\circle*{0.12}}
\put(10.1,2.1){\circle*{0.12}}
\put(10.4,1.8){\circle*{0.12}}
\put(10.7,1.5){\circle*{0.12}}
\put(9.6,1.8){\makebox(0,0)[t]{{\small $V$}}}

\end{picture}

\caption{The digraphs $A_1$ and $T = {\rm desc}(b)$.}
\label{exsigma}
\end{center}
\end{figure}
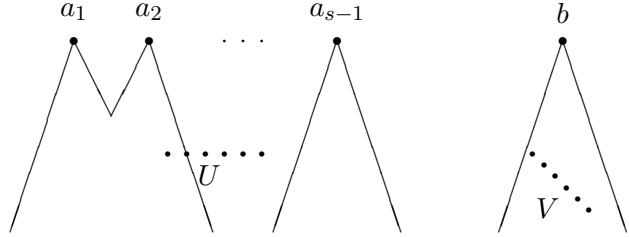
 
By the induction hypothesis, $A_1 \in {\cal D}$ since $A_1 \leq A \in {\cal C}_n$ and $A_1$ has $s-1$ generators. Let $P$ be the set of all 
$q$-element subsets of $U$. 

\begin{lemma}\label{4.5} There is $B \in {\cal D}$ containing $A_1$ such that every member of $P$ with at most $n-2$ common predecessors in $A_1$ has 
at least one common predecessor in $B$ which does not lie in $A_1$. 
\end{lemma}
\begin{proof}
Let $M$ be the greatest number of common predecessors in $A_1$ of an element of $P$. Note that $M \leq n-1$ since 
${\cal D} \subseteq {\cal C}_n$, and recall that ${\cal T}_{n-1} \in {\cal D}$. Now apply Lemma \ref{4.3} with $N = n-1$ to obtain $B \in 
{\cal D}$ containing $A_1$ and such that every $p \in P$ has at least $n-1$ common predecessors in $B$. So for $p \in P$ with at most $n-2$ common predecessors in $A_1$, there is at least one common predecessor of $ p$  in $B$  which does not lie in $A_1$. \end{proof}

\medskip
Let $T' \geq T$ be a copy of $T$ with root $z$ such that $(z,b)$ is  a directed edge; let $b' \neq b$ be  another successor of $z$. 

We now find  independent subsets $U' \cup I$ and $V' \cup J$ of $B$ and $T'$ respectively, with ${\rm desc}(U) \subseteq {\rm desc}(U')$ and ${\rm desc}(V') \subseteq {\rm desc}(V)$, such that $I$ is a complement to $U$ in $A_1$ and $J \subseteq {\rm desc}(b')$, together with an isomorphism from ${\rm desc}(U' \cup I)$ to ${\rm desc}(V' \cup J)$ which takes $I$ to $J$ and extends $f$.

Indeed, if $n = 2$ we let $U':=U$ and $V':=V$. Now suppose $n \geq 3$ and let $P'$ be the subset of $P$ consisting of all $q$-element sets $p 
\subseteq U$ with at least one and at most $n-2$ common predecessors in $A_1$, and such that the image $f(p)$ in $V$ has a common predecessor in $T'$. 
By Lemma \ref{4.5}, $p$ has a common predecessor $w_p$ in $B \setminus A_1$. Let $w_{f(p)}$ be the common predecessor of $f(p)$ in $T'$ and define 
\[
U':= \left( U \setminus \bigcup P' \right) \cup \{w_p\mid p \in P'\}
\]
and 
\[ 
V':=\left( V \setminus  \bigcup  f( P')\right) \cup \{w_{f(p)}\mid p \in P'\}
\]
By Lemma \ref{4.2}, $U'$ and $V'$ are independent subsets of $B$ and $T'$ respectively and the natural extension $F$ of $f$ which takes $w_p$ to 
$w_{f(p)}$ is an isomorphism from ${\rm desc}(U')$ to ${\rm desc}(V')$. In either case ($n = 2$ or $n \geq 3$), by Lemma \ref{4.1}, $U$ has a 
complement $I$ in $A_1$ and by Lemma \ref{4.2}, $U' \cup I$ is an independent set. Now let $J$ be an independent subset of ${\rm desc}(b')$ with 
$\left|J\right| = \left| I \right|$.  Since ${\rm desc}(b) \cap {\rm desc}(b')=\varnothing$, $V' \cup J$ is an independent subset of $T'$.  So $U' 
\cup I$ and $V' \cup J$ are independent subsets of the same size and there is an isomorphism $\overline F$ from ${\rm desc}(U' \cup I)$ to 
${\rm desc}(V'\cup J)$ extending $F$ and taking $I$ to $J$. 

\smallskip

By $\leq$-amalgamation, there are $D \in {\cal D}$ and $\leq$-embeddings $g_1 : B \rightarrow  D$, $g_2: T' \rightarrow D$ such that $g_1(y) = g_2 
({\overline F}(y))$ for all $y \in {\rm desc}(U'\cup I)$, where we may assume that $g_1$ is the identity map. As we now show, the point of the 
construction is that by extending {\em before} we amalgamate, we have ensured that in this amalgamation, unwanted identifications are avoided. 
 
\medskip
 
\begin{lemma}\label{4.6} $ A_1\cap g_2({\rm desc}(b))={\rm desc}(U)$.  \end{lemma}
\begin{proof} We have ${\rm desc}(U)=g_2({\rm desc}(V))$ since ${\overline F}_{\mid {\rm desc}(U)}=f$. As ${\rm desc}(V) \subseteq {\rm desc}(b)$, it follows that ${\rm desc}(U) \subseteq A_1 \cap g_2({\rm desc}(b))$. Now suppose for a contradiction that  there are vertices $\gamma \in A_1 \setminus {\rm desc}(U)$, $\gamma' \in {\rm desc}(b) \setminus {\rm desc}(V)$ such that $\gamma=g_2(\gamma')$. 
 
We first show that ${\rm desc}(\gamma) \setminus {\rm desc}(U)$ is finite. Indeed, suppose $a\in A_1$ is such that ${\rm desc}(a) \cap {\rm desc}(I) 
\neq \varnothing$. Then $a \neq g_2(\gamma'')$ for any $\gamma'' \in {\rm desc}(b) \setminus {\rm desc}(V)$ since ${\rm desc}(I) = g_2({\rm desc}(J)) 
\subseteq g_2({\rm desc}(b'))$ and ${\rm desc}(b') \cap {\rm desc}(b)=\varnothing$. So ${\rm desc}(\gamma) \cap {\rm desc}(I) = \varnothing$, and 
${\rm desc}(\gamma) \setminus {\rm desc}(U) = {\rm desc}(\gamma) \setminus {\rm desc}(U\cup I)$ is finite. 

Now we show that there is a $q$-element subset $p$ of $U \cap {\rm desc}({\gamma})$ with a common predecessor in 
${\rm desc}(\gamma)$. Choose $u \in U \cap {\rm desc}({\gamma})$ at maximal distance from $\gamma$, and let $y$ be the 
predecessor of $u$ in ${\rm desc}(\gamma)$ (note that $y \in  A_1$). Since ${\rm desc}(\gamma) \setminus {\rm desc}(U)$ is 
finite, ${\rm desc}(y) \setminus {\rm desc}(U)$ is finite. So if $u'$ is another successor of $y$, ${\rm desc}(u') \setminus 
{\rm desc}(U)$ is finite and our choice of $u$ implies that $u' \in U \cap {\rm desc}({\gamma})$. Thus we can take $p$ 
to be the set of successors of $y$. 
\smallskip

Now we finish off the proof of Lemma \ref{4.6}. Since $\gamma = g_2(\gamma')$, the $q$-element subset $f(p)$ of ${\rm desc}(\gamma') \cap V$ has a 
common predecessor, $y'$ say, in ${\rm desc}(\gamma')$ and $y=g_2(y')$. If $p$ has  $n-1$ predecessors in $A_1$,  then there is a copy of ${\cal T}_n$ 
in $A$ because $A$ is the free amalgam of $A_1$ and $T$ over ${\rm desc}(U)={\rm desc}(V)$. This is a contradiction. Therefore $p$ has at most $n-2$ 
common predecessors in $A_1$ and this means that $p \in P'$. It follows that $y' = w_{f(p)}$ since a $q$-element set of vertices of $T'$ has at most 
one common predecessor in $T'$ as $T'$ is a tree. Now as $w_p=g_2(w_{f(p)})$, we have $y=g_2(y')=g_2(w_{f(p)})=w_p$. This is a contradiction since 
$w_p \in B\setminus A_1$ and $y \in A_1$. \end{proof}

\medskip

We have therefore shown that, $A_1 \cup g_2({\rm desc}(b))$ as a subdigraph of $D$ is isomorphic to $A$.  So $A$ embeds in $D$ 
and therefore $A \in {\cal D}$. This completes the proof that ${\cal D} = {\cal C}_n$. \end{proof}

\bigskip

Finally suppose ${\cal D} \subseteq {\cal C}$ is a $\leq$-amalgamation class and  ${\cal T}_n \in {\cal D}$ for all 
$n \geq 1$. A similar argument as in the 
above proof can be used to show that any $A \in {\cal C}$ lies in $\cal D$. The two important points in that proof which we need to modify slightly, 
are the choice of the digraph $B$  and of the subset $P'$ of $P$. We want a digraph $B \in {\cal D}$ 
containing $A_1$ such that every $q$-element set of vertices of $U$ has at least $M+1$ common predecessors in $B$, where $M$ is the greatest number of 
common predecessors of $p$ in $A_1$ as $p$ ranges over $P$. For this we apply Lemma \ref{4.3} to $A_1$ with $N:=M+1$. In this case it will 
follow that for \textit{every} $p$ in $P$, there is at least one common predecessor of $p$ in $B$ which does not lie in $A_1$. We then take $P'$ to be the subset of $P$ consisting of all 
$q$-element sets $p$ which have at least one common predecessor in $A_1$ and such that $f(p)$ has a common predecessor in 
$T'$. The remainder of the 
argument follows similarly, except that when showing that $A_1 \cap g_2({\rm desc}(b)) ={\rm desc}(U)$, there is only one case to consider since for every 
$p$ in $P'$ there is a vertex $w_p \in B \setminus A_1$. We deduce the following.

\begin{proposition}
Let ${\cal D} \subseteq {\cal C}$ be a $\leq$-amalgamation class with ${\cal T}_n \in {\cal D}$ for all $n \geq 1$. Then ${\cal D} = {\cal C}$.
\end{proposition}\smallskip

We have therefore shown that 

\begin{theorem}
Any $\leq$-amalgamation class ${\cal D} \subseteq {\cal C}$ is equal to  $\cal C$ or to ${\cal C}_n$ for some $n \geq 2$. 
\end{theorem}

This means that if $D$ is a countable descendant-homogeneous digraph whose descendant set is isomorphic to $T$, then $D\cong D_n$ for some $n \in \{2, \ldots, \infty\}$.

\section{A general construction}\label{general}

\subsection{Descendant sets}

In this subsection we prove the following.

\begin{theorem}\label{desccond}
Suppose $\Gamma$ is a countable digraph. Then there is a countable, vertex transitive, descendant-homogeneous digraph $M$ in which all descendant sets are isomorphic to $\Gamma$ if and only if the following conditions hold:
\begin{enumerate}
\item [{\rm (C1)}]${\rm desc}(u)\cong \Gamma $ for all $u\in \Gamma$;

\item [{\rm (C2)}] If $X$ is a finitely generated subdigraph of $\Gamma$ then the subgroup of automorphisms of $X$ which extend to automorphisms of $\Gamma$ is of countable index in $\Aut(\Gamma)$ .
\end{enumerate}
\end{theorem}

For one direction of this, suppose $M$ is  a vertex transitive, descendant-homoge\-neous digraph. The descendant sets of vertices in $M$ are all isomorphic to a fixed digraph $\Gamma$, so (C1)
 holds. Condition (C2) is a special case of (4$'$) in Remark \ref{rem2}, so follows from Theorem \ref{2.1} and Remark \ref{rem2}.
 
 We now prove the converse. So for the rest of this subsection, suppose that $\Gamma$ is a countable digraph which satisfies conditions (C1) and (C2). Let ${\cal C}_{\Gamma}$ be the class of digraphs $A$ satisfying the following conditions:

\begin{enumerate}
\item[(D1)]  ${\rm desc}(a)$ is isomorphic to $\Gamma$, for all $a \in A$; 
\item[(D2)] $A$ is finitely generated;
\item[(D3)] for $a,b \in A$, the intersection ${\rm desc}(a)\cap {\rm desc}(b)$ is finitely generated.
\end{enumerate}

Then ${\cal C}_{\Gamma}$ is closed under isomorphism and taking finitely generated descendant-closed substructures. Moreover, it is easy to see that if $A \leq B_1, B_2 \in \C_\Gamma$ and $A$ is finitely generated, then the free amalgam of $B_1$ and $B_2$ over $A$ is in $\C_\Gamma$. Thus, Theorem \ref{desccond} will follow once we verify that the countability conditions in (1) and (4) of Theorem \ref{2.1} hold for $\C_\Gamma$. The following lemma will suffice.

\begin{lemma} \label{countable} Suppose $A \in \C_\Gamma$. Then there are only countably many isomorphism types of $\leq$-embeddings $f : A \to B$ with $B \in \C_\Gamma$. 
\end{lemma}

Once we have this, taking $A = \emptyset$ (or $A = \Gamma$) gives that $\C_{\Gamma}$ contains only countably many isomorphism types; for fixed $A, B \in \C_\Gamma$, the lemma gives condition (4) of Theorem \ref{2.1}.

\medskip

\begin{proof}
We say that a $\leq$-embedding $f : A \to B$ with $A, B \in \C_\Gamma$ is an \textit{$n$-extension} if $B$ can be generated by $f(A)$ and at most $n$ extra elements. We prove by induction on $n$ that for every $A \in \C_\Gamma$ there are only countably many isomorphism types of $n$-extensions of $A$. 

Suppose $f : A \to B$ is a $1$-extension (with $A, B \in \C_\Gamma$). Let $b \in B$ be such that $B$ is generated by $f(A)$ and $b$ and let $C = f(A) \cap {\rm desc}(b)$. It follows from property  (D3) in $B$ and finite generation of $A$, that $C$ is finitely generated. Moreover, as each of $f(A)$ and ${\rm desc}(b)$ is descendant-closed in $B$, we have that $B$ is the free amalgam of $f(A)$ and ${\rm desc}(b)$ over $C$. Choose an isomorphism from $\desc(b)$ to $\Gamma$ and let $h$ be the restriction of this to $C$ and $g : f^{-1}(C) \to \Gamma$ be given by $g = h \circ f$. Then, in the notation of Remark \ref{freedef}, we have an isomorphism from $B$ to $A \ast_{g} \Gamma$ and therefore $f$ is isomorphic to a $1$-extension $A \to A \ast_g \Gamma$ for some finitely generated $D \leq A$ and $\leq$-embedding $g : D \to \Gamma$. 

There are countably many possibilities for $D$ and the image $g(D)$ here (as $D$ is finitely generated), so it will suffice to show that there are only countably many isomorphism types of $A \ast_{g} \Gamma$ with $g : D \to \Gamma$ having fixed domain $D$ and image $E \leq \Gamma$. If $g_1, g_2 : D \to \Gamma$ have image $E$ then $g_1 \circ g_2^{-1}$ gives an automorphism of $E$. This extends to an automorphism  of $\Gamma$ if and only if there is an isomorphism between the extensions $g_i : A \to A\ast_{g_i} \Gamma$. Thus, the isomorphism types here are in one-to-one correspondence with the cosets in $\Aut(E)$ of the subgroup of automorphisms which extend to automorphisms of $\Gamma$. So there are only countably many isomorphism types, by (C2).

This proves that there are countably many isomorphism types of $1$-extensions of $A$. For the inductive step, we can take countably many representatives $f_j : A \to B_j'$ (for $j \in \N$) of the isomorphism types of $(n-1)$-extensions of $A$, and representatives $h_{jk} : B_j' \to B_{jk}'$ of the $1$-extensions of $B_j'$ (for $j, k \in \N$). We claim that any $n$-extension $f : A \to B$ is isomorphic to some $h_{jk} \circ f_j : A \to B_{jk}'$. Indeed, let $f(A)\leq B_1 \leq B$ be such that $B_1$ is generated by $f(A)$ and $n-1$ elements, and $B$ is generated by $B_1$ and one extra element. So we can write $f = i\circ g$ where $g : A \to B_1$ is an $(n-1)$-extension and $i : B_1 \to B$ is a $1$-extension. There is $j \in \N$ and an isomorphism $h : B_j' \to B_1$ with $h \circ f_j = g$. We can then find $k \in \N$ and an isomorphism $p : B_{jk}' \to B$ with $i\circ h = p \circ h_{jk}$. Then $p\circ h_{jk}\circ f_j = g \circ i = f$, as required. 
\end{proof}

\medskip

It then follows by  Theorem \ref{2.1} that  the Fra\"{\i}ss\'e limit $D_{\Gamma}$ of $(\C_\Gamma, \leq)$ is  a countable descendant-homogeneous digraph with $\C_\Gamma$  as its class of finitely generated $\leq$-subdigraphs. Vertex transitivity follows from (C1). 

\subsection{Examples and further remarks}

In this subsection we show that a class of digraphs $\Gamma$ arising in \cite{amato1} in the context of highly arc transitive digraphs satisfy the conditions in Theorem \ref{desccond} and therefore arise as the descendant sets in descendant-homogeneous digraphs. We begin by reviewing some of the results of \cite{amato1} and related papers.

The paper \cite{amato1} studies highly arc transitive digraphs of finite out-valency and gives conditions which the descendant set $\Gamma$ of a vertex in such a digraph must satisfy. In particular:

\begin{theorem} Suppose $\Gamma$ is the descendant set of a vertex $\alpha$ in an infinite  highly arc transitive digraph $D$ of finite out-valency. Then the following properties hold:
\begin{enumerate}
\item[(T1)]  $\Gamma = {\rm desc}(\alpha)$ is a rooted digraph with finite out-valency and ${\rm desc}^s(\alpha) \cap {\rm desc}^t(\alpha) = \varnothing$ whenever $s\neq t$.
\item[(T2)] ${\rm desc}(u)\cong \Gamma $ for all $u\in \Gamma$.
\item[(T3)] $\Aut(\Gamma)$ is transitive on $\desc^s(\alpha)$, for all $s$.
\item[(T4)] There is a natural number $N=N_{\Gamma}$ such that for $l >N$ and $x, a \in \Gamma$, if $b \in {\rm desc}^l(x)\cap \desc^1(a)$ , then $a \in {\rm desc}(x)$.
\end{enumerate}
\end{theorem}

\begin{proof} Properties (T2) and (T3) follow immediately from high arc transitivity of $D$. Property (T1) is proved in Lemma 3.1 of \cite{amato1} and (T4) is deduced from (T1), (T2), (T3) in (\cite{amato1}, Proposition 4.7(a)). 
\end{proof}

\medskip

\begin{remark} \rm The paper \cite{amatode} shows that there are only countably many isomorphism types of digraphs $\Gamma$ satisfying properties (T1, T2, T3). In fact, the same is true with (T3) replaced by the weaker:
\begin{enumerate}
\item[(G3)] There is a natural number $k$ such that if $\ell \geq k$ and $x \in \desc^{\ell}(\alpha)$ and $\beta \in \desc^1(\alpha)$, then $\desc(\beta) \cap \desc(x) \neq \emptyset$ implies $x \in \desc(\beta)$.
\end{enumerate}
Moreover, these (T1, T2, G3) imply (T4). 
See Corollary 1.5 and Lemma 2.1 of \cite{amatode} for proofs.
\end{remark}

Explicit examples $\Gamma(\Sigma, k)$ of digraphs satisfying (T1, T2, T3) (and which are not trees) are constructed in Section 5 of \cite{amato1} and  constructions of  highly arc transitive, but not descendant-homogeneous, digraphs with these as descendant sets are  given in \cite{amato1} and \cite{amato2}. The construction we give here (using Theorem \ref{desccond}) gives a highly arc transitive, descendant-homogeneous digraph with descendant set $\Gamma(\Sigma, k)$ (and which does not have property $Z$). Indeed, it is a slightly curious corollary of the results of this section that if $\Gamma$ is a digraph of finite out-valency which is the descendant set of a vertex in an infinite, highly arc transitive digraph, then there is a descendant-homogeneous, highly arc transitive digraph which has $\Gamma$ as its descendant set.

\begin{corollary}\label{extension}
Suppose $\Gamma$ is a digraph of finite out-valency which satisfies conditions (T1, T2, T4). Then there is a countable, vertex transitive, descendant-homogeneous digraph in which all descendant sets are isomorphic to $\Gamma$.
\end{corollary}

\medskip

\begin{proof} We use Theorem \ref{desccond}. The digraph $\Gamma$ satisfies condition (C1) of this, by assumption (T2). So it remains to show that $\Gamma$ satisfies (C2).

Let $X$ be a finitely generated subdigraph of $\Gamma$ with minimal generating set $\{x_1, \ldots,x_k\}$. Let  $N=N_{\Gamma}$ as in (T4). We will show that any automorphism of $X$ fixing pointwise the union of balls $Y:=\bigcup_{i=1}^{k}B^N(x_i)$  extends   to an automorphism of $\Gamma$. As these automorphisms form a subgroup of finite index in $\Aut(X)$, condition (C2) follows.

Let $a,b \in \Gamma$. We first observe that if $b \in X \setminus Y$ and $a$ is a predecessor of $b$ in $\Gamma$, then $a \in X$. Indeed, $b \in {\rm desc}^l(x_i)$ for some  $l>N$ and $ i \in \{1, \ldots, k\}$. Then by definition of $N$, $a \in {\rm desc}(x_i)$. Since ${\rm desc}(x_i)  \subseteq X$, it follows that $a \in X$.

Let  $\gamma$ be an automorphism of $X$ which fixes $Y$ pointwise. Define $\theta=\gamma \cup id_{\Gamma \setminus X}$. To prove $\theta$ is an automorphism of $\Gamma$ we must show  that $\theta$ preserves edges and non-edges. For $u \in (\Gamma\setminus X) \cup Y$, $\theta u =u$ and for $u \in X$, $\theta u=\gamma u$. So for $a,b \in (\Gamma \setminus X) \cup Y$, we have $\theta(a,b)=(\theta a, \theta b)=(a,b)$. Similarly, ${\theta}$ preserves edges and non-edges when $a,b \in X$ as in this case, $\theta (a,b)=\gamma(a,b)$.  Now suppose $a \in \Gamma \setminus X$ and $b \in X \setminus Y$. The image $\theta(a,b)=(\theta a, \theta b)=(a,\gamma b)$. Since $\gamma$ preserves $Y$, $\gamma b \in X\setminus Y$. Then by the observation above, $(a,b)$ and $(a,\gamma b)$ are non-edges. 
\end{proof}

\end{document}